\documentclass{crm-eca}

\usepackage{Irtatca}  % Provides workshop title and series information. Do not comment.
\usepackage{amssymb}

\usepackage{mathrsfs}

\usepackage{graphicx}

\usepackage[matrix,arrow,curve]{xy}
\usepackage{lscape}
\usepackage{color}
\newtheorem{theorem}{Theorem}
\newtheorem{lemma}[theorem]{Lemma}
\newtheorem{corollary}[theorem]{Corollary}
\newtheorem{proposition}[theorem]{Proposition}

\theoremstyle{definition}
\newtheorem{definition}[theorem]{Definition}

\theoremstyle{remark}
\newtheorem{remark}[theorem]{Remark}

\begin{document}
\papertitle[Proalgebraic crossed modules]{Proalgebraic crossed modules of quasirational presentations}
% Use \\ if the title consists of multiple lines
\papertitle{Proalgebraic crossed modules of quasirational presentations}
\paperauthor{Andrey Mikhovich}
\paperaddress{Moscow State University} % Use \newline if the address consists of multiple lines
\paperemail{mikhandr@mail.ru}

%\paperthanks{}

\makepapertitle

%%%A 3-4 lines abstract is mandatory
\Summary{
For every quasirational (pro-$p$)relation module $R$ we construct the so called $p$-adic rationalization, which is the pro-fd-module
$\overline{R}\widehat{\otimes}\mathbb{Q}_p= \varprojlim R/[R,R\mathcal{M}_n]\otimes\mathbb{Q}_p$, and prove the isomorphism $\overline{R}\widehat{\otimes}\mathbb{Q}_p\cong\overline{R^{\wedge}_w}(\mathbb{Q}_p)$, where $\overline{R^{\wedge}_w}(\mathbb{Q}_p)$ stands for the rational points of the abelianization of the continuous $p$-adic Malcev completion of $R$. We show how $\overline{R^{\wedge}_{w}}$ embeds into a sequence which arises from a certain prounipotent crossed module. The latter can be seen as concrete examples of proalgebraic homotopy types.

We provide the Identity Theorem for pro-$p$-groups, giving a positive feedback to the question of Serre.}

%    Text of article.

%    Numbered section headings
\section{Schematization and proalgebraic crossed modules}
\
What is homotopy theory? After Quillen we may regard this as formal settings of model categories and their equivalences. The Quillen-equivalence between homotopy categories of compactly generated Hausdorff spaces and simplicial sets is the brightest example (so the importance of simplicial sets). However Grothendieck's meditative dream was that combinatorial homotopy theory lies not far from geometry in its schematic reincarnation (just as homotopical invariants of smooth manifolds). Recently this idea was realized in B.To\"{e}n's theory of schematic homotopy types. The constructive form gives for a pair $(X,k)$ (where $X$ be a connected simplicial set and $k$ be a field) some schematic homotopy type, which is a simplicial proalgebraic group $(GX)^{\wedge}_{alg}$, where $G$ is a Kan loop-group functor, $(GX)^{\wedge}_{alg}$ is the proalgebraic completion of a free simplicial group $GX$.
 Lets look at old problems of two-dimensional combinatorial homotopy theory using schematic glasses. First, we apply Kan loop-group functor to connected two-dimensional simplicial set and obtain a free simplicial group degenerated in dimensions greater than one. Kan results gives a $CW$-basis for this free simplicial group, so we get some free simplicial group also degenerated in dimensions greater than one
\begin{equation}
\xymatrix{
& \ar@<1ex>[r]\ar@<-1ex>[r] \ar[r] & {F(X\cup Y)} \ar@<0ex>[r]^{d_0}\ar@<-2ex>[r]^{d_1} & F(X) \ar@<-2ex>[l]_{s_0} }, \label{2}
\end{equation} where $d_0, d_1, s_0$, on $x \in X, y \in Y, r_y \in F(X)$ defined by:
$d_1(x)=x,  d_1(y)=r_y,$
$d_0(x)=x,  d_0(y)=1,$
$s_0(x)=x.$
The corresponding 2-reduced simplicial group \begin{equation*}
\xymatrix{
& {F(X\cup Y)} \ar@<0ex>[r]^{d_0}\ar@<-2ex>[r]^{d_1} & F(X) \ar@<-2ex>[l]_{s_0} }
\end{equation*} is the standard object for study in combinatorial group theory. We use the proalgebraic completion to jump into reduced schematic homotopy types \begin{equation*}
\xymatrix{
& {F(X\cup Y)^{\wedge}_{alg}} \ar@<0ex>[r]^{d_0}\ar@<-2ex>[r]^{d_1} & F(X)^{\wedge}_{alg} \ar@<-2ex>[l]_{s_0} }
\end{equation*} But the groups of $k$-points $F^{\wedge}_{alg}(k)$ here are too big for practical purposes, and we need to find reasonable approximation, which is sufficiently reach for 2-dimensional combinatorial homotopy. But old constructions of Quillen ("Quillen formula" \cite[Cor.21]{Vez}) and Magnus (see below) contain the helpful hint.
\begin{remark}[Magnus embedding] Several important homotopical invariants in the theory of groups and geometric topology defined using embedding of the group ring of a free group $F$ into the algebra of formal power series on noncommutative indeterminates:
$\mu: \mathbb{Z}F(x_1,...,x_n)\hookrightarrow \mathbb{Z}\ll t_{x_1},...,t_{x_n}\gg, \quad \mu(x_i)=1+t_{x_i},$ where $\mu(x_i^{-1})=1-t_{x_i}+t_{x_i}^2-...$
Indeed, integral coefficients do not play any principal role and we use embedding $F(x_1,...,x_n)\hookrightarrow k\ll t_{x_1},...,t_{x_n}\gg$ with coefficients in any field $k$. Now Pontryagin- Van Kampen duality gives a possibility to look at defining relations as at linear functions on a certain commutative Hopf algebra \cite[3]{Vez}.
\end{remark}
This all means, by the practical reasons, we must restrict ourselves to the prounipotent completions and to corresponding prounipotent homotopy types. In the case of 2-homotopy we will work with 2-reduced simplicial prounipotent groups
\begin{equation}
\xymatrix{
& {F(X\cup Y)^{\wedge}_u} \ar@<0ex>[r]^{d_0}\ar@<-2ex>[r]^{d_1} & F(X)^{\wedge}_u \ar@<-2ex>[l]_{s_0} }\label{eq2}
\end{equation}
Let $R_u=d_1(ker(d_0))$ be the corresponding Zarissky normal closure of defining relations and $G_u:=F_u(X)/R_u$.
\textbf{We will consider unless otherwise stated only finite presentations} $\mathbf{(|X|< \infty, |Y|< \infty )}$, because we still have no sufficient understanding of what $F(X)^{\wedge}_u$ is in infinite case.
The choice of prounipotent affine group schemes is not occasional, there is an entirely schematic explanation. In fact $(G/H)(k)\cong G(k)/H(k)$ ($char(k)=0$) and so group-theoretic settings are compatible with schematic. The simplicial identity $d_0 s_0=id_{F_u(X)}$ (saving notions) implies that (as in the discrete case) we have the formula $F_u(X\cup Y)\cong Ker d_0 \leftthreetimes s_0 F_u(X).$ Now one can collect necessary homotopical information from the study of $Ker d_0 \xrightarrow{d_1} F_u(X)$ with the action of $F_u(X)$ through $s_0$ by conjugation on $Ker d_0$. $(Ker d_0, F_u(X), d_1)$ is a particular example of a prounipotent precrossed module. The consruction of prounipotent (pre)crossed modules in general (and their category as well) is standard and we refer the reader to consult with \cite{BH}. To obtain a bridge between ordinary combinatorial group theory and prounipotent (pre)crossed modules we need the following definition due to \cite{HM,Pri,Pri2012}.

The group of $\mathbb{Q}_p$-points of any affine group scheme $G$ has the $p$-adic topology. Indeed, \cite{DM} shows that $G$ can be expressed as a filtered inverse limit $G=\varprojlim G_{\alpha}$
of linear algebraic groups. Each $G_\alpha(\mathbb{Q}_p)$ has a canonical $p$-adic topology induced by the embedding $G_{\alpha}\hookrightarrow GL_n$. Define the topology on $G(\mathbb{Q}_p)$ by
$G(\mathbb{Q}_p)=\varprojlim G_{\alpha}(\mathbb{Q}_p).$
\begin{definition}
Fix a topological group $G$ (pro-$p$-topology in our further considerations). Define $p$-adic Malcev completion of $G$ by a universal diagram where $\rho$ is continuous Zarissky-dense homomorphism of $G$ into $\mathbb{Q}_p$- points of a prounipotent affine group $G_w^{\wedge}$
$$\xymatrix@R=0.5cm{
                &         G^{\wedge}_w(\mathbb{Q}_p)  \ar[dd]^{\tau}     \\
 G \ar[ur]^{\rho} \ar[dr]_{\chi}                 \\
                &         H(\mathbb{Q}_p)              }$$
We require that for every continuous Zarissky-dense homomorphism $\chi$ into $\mathbb{Q}_p$- points of a prounipotent affine group $H$ there is a unique homomorphism $\tau$ of prounipotent groups, making the diagram commutative.
\end{definition}

\begin{remark} Finite cardinality of $Z$ is a sufficient condition to identify a free prounipotent group $F_u(Z)$ with $F(Z)^{\wedge}_w$ and $F(Z)^{\wedge}_u$. However we use different letters "u" or "w" to emphasize the source ("w"- in the case of pro-$p$-continuous Malcev completions, "u"- in the case of prouniptent presentations). If \eqref{eq2} comes from a pro-$p$-presentation (since any homomorphism of finitely generated pro-$p$-group into $\mathbb{Q}_p$-points of prounipotent group is continuous \cite[Lemma A.7]{HM}) we also have the isomorphism $G_u\cong G^{\wedge}_w$.
\end{remark}

\begin{definition}[\cite{Mikh}]
Free prounipotent (pre)crossed with a basis $Y\in G_2(\mathbb{Q}_p)$ be a prounipotent (pre)crossed module $(G_2,G_1,d)$ such that $Y$  generates a Zarissky dense $G_1(\mathbb{Q}_p)$-subgroup in $G_2(\mathbb{Q}_p)$ with the following universal property:

for any prounipotent (pre)crossed $(G'_2,G'_1,d')$ and any subset $\nu(Y) \in G'_2(\mathbb{Q}_p)$ and a function $\nu:Y \rightarrow G'_2(\mathbb{Q}_p),$ with Zarissky dense $G'_1(\mathbb{Q}_p)$-group closure and any epimorphism of prounipotent groups $f:G_1\twoheadrightarrow G'_1$, with $$fd(\mathbb{Q}_p)(Y)=d'(\mathbb{Q}_p)\nu(Y),$$ then there exists a unique homomorphism of prounipotent groups $h:G_2\rightarrow G'_2$ such that $h(\mathbb{Q}_p)(Y)=\nu(Y)$ and the pair $(h,f)$ is the homomorphism of (pre)crossed modules.

\end{definition}
\begin{lemma}[\cite{Mikh}]
$Kerd_{0} \xrightarrow{d_1} F_u(X)$ is a
free prounipotent precrossed module on $Y$, where $d_1$ comes from a presentation \eqref{eq2}.
\end{lemma}

We construct \cite{Mikh} the prouniptent crossed module $ (C_u,F_u(X),\overline{d})$ of a prounipotent presentation \eqref{eq2} as follows:
$\frac{Kerd_0}{P_u} \xrightarrow{\overline{d}} F_u(X),$ where $P_u=<{}^{\overline{d}(b)}a=bab^{-1}>$ is a Zarissky normal closure of Peiffer commutators and $\overline{d}$ arises from $d_1$ in a presentation \eqref{eq2}. The equality $P_u(A)= [Kerd_0, Kerd_1](A)$ holds for any $\mathbb{Q}_p$-algebra $A$, which certainly implies:

\begin{lemma}[\cite{Mikh}]
There is an isomorphism of prouniptent crossed modules
$(\frac{Kerd_0}{P_u}, F_u(X),\overline{d})\cong(\frac{Kerd_0}{[Kerd_0, Kerd_1]},F_u(X),\overline{d})$
\end{lemma}

\begin{lemma}[\cite{Mikh}]
$R_u$ acts trivially on $\overline{C}_u$ , where $\overline{C}_u$ be a factor of ${C}_u$ by the derived subgroup and hence $\overline{C}_u(\mathbb{Q}_p)$ is a $\mathcal{O}(G_u)^*$-module, where $\mathcal{O}(G_u)^*$ be a Pontryagin- Van Kampen dual to the representing Hopf algebra $\mathcal{O}(G_u)$ of $G_u$.
\end{lemma}

\section{Quasirational presentations}
\textbf{For pro-$p$-groups, fix a prime $p>0$ throughout the paper}
(see \cite
{Se3} for details on pro-$p$-groups). \textbf{
For  discrete groups, $p$ will vary}. Let $G$ be a (pro-$p$)group which has a (pro-$p$)presentation of finite type
\begin{eqnarray}
1 \rightarrow R \rightarrow F \rightarrow G \rightarrow 1 \label{1}
\end{eqnarray}

Let $\overline{R}=R/[R,R]$ be the corresponding relation $G$-module, where $[R,R]$ is a (closed) commutator subgroup (in the
pro-$p$-case). Then denote $\mathcal{M}_n$ the corresponding Zassenhaus $p$-filtration of $F,$ which is defined by the rule $$\mathcal{M}_n=\{f \in F\mid f-1 \in {\Delta}^n_p, \quad {\Delta}_p =ker\{\mathbb{F}_pF \rightarrow \mathbb{F}_p\}\}$$ (see \cite[7.4]{Koch} for details).

\begin{definition} A presentation \eqref{1} is \textbf{quasirational} if for every $n>0$ and each prime $p>0$ the $F/R\mathcal{M}_n$-module $R/[R,R\mathcal{M}_n]$ has no $p$-torsion ($p$ is fixed for pro-$p$-groups and runs through all primes $p>0$ and corresponding $p$-Zassenhaus filtrations in the discrete case). The relation modules of such presentations will be called \textbf{quasirational relation modules}.
\end{definition}
\begin{proposition}[\cite{Mi}] Suppose \eqref{1} is a presentation of a pro-$p$-group $G$ with a single defining relation, then \eqref{1} is quasirational.
\end{proposition}

Quasirational presentations  may be studied by passing to rationalized completions $$\overline{R}\widehat{\otimes} \mathbb{Q}_p=\varprojlim R/[R,R\mathcal{M}_n]\otimes \mathbb{Q}_p$$ (since $\varprojlim$ is left exact for quasirational (pro-$p$)presentations we have an embedding of abelian groups $\overline{R}\hookrightarrow \overline{R}\widehat{\otimes} \mathbb{Q}_p$). In a similar manner $\mathbb{Q}_p$-points of the abelianized continuous prounipotent completion of $$C^{\wedge}_w(\mathbb{Q}_p)=(\frac{kerd_0}{[kerd_0,kerd_1]})^{\wedge}_w(\mathbb{Q}_p)\cong (\mathbb{Q}_pG^{\wedge}_p)^{|Y|}$$
have the structure of topological $G^{\wedge}_p$-module ($\overline{R}\widehat{\otimes} \mathbb{Q}_p$ has a similar $G^{\wedge}_p$-module structure). First, we obtain $$\overline{R}\widehat{\otimes} \mathbb{Q}_p\cong \overline{R^{\wedge}_w}(\mathbb{Q}_p)$$ \cite{Mikh} and now the following description (in the spirit of Gasch\"{u}tz theory \cite{Gru}) is crucial:

\begin{theorem}[Infinite Gasch\"{u}tz Theory \cite{Mikh}]

Let \eqref{1} be finite $QR$-(pro-$p$)presentation, then there is a commutative diagram of abelian prounipotent groups, where on $\mathbb{Q}_p$-points the upper arrow is a homomorphism of topological $G^{\wedge}_p$-modules, the lower is a homomorphism of $\mathcal{O}(G_u)^*$-modules, vertical arrows are $G^{\wedge}_p$-homomorphisms on Zarissky dense subgroups $(\frac{kerd_0}{[kerd_0,kerd_1kerd_0]})^{\wedge}_p$ and $\overline{R^{\wedge}_p}$ in $\overline{C^{\wedge}_w}(\mathbb{Q}_p)$ and
$\overline{R^{\wedge}_w}(\mathbb{Q}_p)$ correspondingly
$$\xymatrix{
  \overline{C^{\wedge}_w} \ar[d]_{\kappa} \ar[r]^{\gamma} & \overline{R^{\wedge}_w} \ar[d]^{\tau} \\
  \overline{C_u} \ar[r]^{\mu} & \overline{R}_u   }$$ we denote $\overline{C^{\wedge}_w}, \overline{C_u}, \overline{R^{\wedge}_w}, \overline{R_u}$ are the abelianizations of the corresponding prounipotent groups.
  $$\overline{C_u}(\mathbb{Q}_p)\cong {\mathcal{O}(G_u)^*}^{|Y|}\cong \mathbb{Q}_p\ll G_u\gg^{|Y|}\cong \widehat{\mathcal{U}}\mathcal{P}\mathcal{O}(G_u)^*,$$ where $\widehat{\mathcal{U}},\mathcal{P}$ are the functors of completed universal enveloping algebra and Lie algebra of primitive elements of complete Hopf algebra,
$\mathbb{Q}_p\ll G_u\gg:=End_{G_u}(\mathcal{O}(G_u)).$
\end{theorem}

\begin{remark} The celebrated Lyndon Identity theorem states, that relation modules of one-relator presentations of discrete groups are induced from cyclic subgroups, i.e. let \eqref{1} be a one-relator, then $$\overline{R}=R/[R,R]\cong \mathbb{Z}\otimes_{\langle u\rangle} \mathbb{Z}G,$$ where $R=(u^m)_F$  and $u$ is not a proper power. We provide the following analog for pro-$p$-groups with a single defining relation (i.e. for pro-$p$-groups $G$: $dim_{\mathbb{F}_p}H^2(G,\mathbb{F}_p)=1$), this is a problem \cite[10.2]{Se2}.
\end{remark}
\begin{corollary}[Identity Theorem for pro-$p$-groups \cite{Mikh2016}]
Let \eqref{1} be a one-relator finitely generated pro-p-group, then there is an isomorphism of topological $\mathcal{O}(G_u)^{\ast}-$modules $$\overline{R_u} \cong \mathcal{O}(G_u)^{\ast}$$

In particular, if $G$ embeds naturally into $\mathbb{Q}_p$-points of its prounipotent completion $G\hookrightarrow G^{\wedge}_u(\mathbb{Q}_p)$, then $cd(G)=2$.

\end{corollary}
\begin{proof} There are two cases:

1. The corresponding prounipotent presentation is degenerated, i.e. $G_u$ is a free prounipotent, then we prove the isomorphism $\overline{R}_u(\mathbb{Q}_p)\cong \mathcal{O}(G_u)^*$ directly by showing that $R_u$ is the free precrossed $G_u$-module \cite{Mikh2016}.

2. In the case of proper prounipotent presentations \cite[3.10]{LM1982} we use (in \cite{Mikh2016}) the modern version of Hochschild cohomology of affine groups \cite{Jan} (a concise introduction may be found in \cite{Ha2015}) to generalize the results from \cite{LM1982,LM1985} for $k\neq \overline{k}$ and to provide the equality $dim_k H^2(G_u,k)=1$.

Further generalization of \cite[Therem 3.14]{LM1982} implies that $cd(G_u)=2$ and reproving \cite[Prop. 3.13]{LM1985} (using Pontryagin- Van-Kampen duality), we obtain $\overline{R_u}(\mathbb{Q}_p)\cong \mathcal{O}(G_u)^*$. 

Finally, Infinite Gasch\"{u}tz Theory and description of aspherical pro-$p$-groups give the final statement.
\end{proof}

 Some ideas and results discussed could be developed for a field of positive characteristic. Anyway the task to compare $p$-adic pro-algebraic and pro-$p$-completions seems very interesting, several beautiful results were obtained in \cite{Pri2012}. From this general perspective quasirationality emphasizes a space where deep interactions between positive and zero characteristics are possible. 
 
 We just mention the construction of "conjurings" \cite{Mikh2016} ("Amitsur-Lewitzky" elements of noncyclic free pro-$p$-groups, but substantially deforming $p$-power structure) elucidating probable absence of "asphericity" \cite{Mel} for one-relator pro-$p$-groups in general.
%\begin{figure}
%\includegraphics{filename}
%\caption{text of caption}
%\label{}
%\end{figure}
% You can also use a package such as \epsfig or \psfig.
% Please use \epsfig rather than \epsf, which has
% become obsolete.
%    Mathematical displays
%    Bibliography.

\end{document}